\documentclass[11pt]{amsart}

\usepackage{amsthm,amsmath,amssymb,amscd,amsfonts,latexsym,}

\theoremstyle{plain}
\newtheorem{theorem}{Theorem}[section]

\newtheorem{proposition}[theorem]{Proposition}
\newtheorem{fundamental fact}[theorem]{Fundamental Fact}

\newtheorem{def-thm}[theorem]{Definition-Theorem}
\newtheorem{lemma}[theorem]{Lemma}

\theoremstyle{definition}
\newtheorem{definition}[theorem]{Definition}

\newtheorem{remark}[theorem]{Remark}
\newtheorem{example}[theorem]{Example}

\newtheorem{conjecture}[theorem]{Conjecture}

\newtheorem*{acknowledgement}{Acknowledgement}

\newcommand{\PP}{\mathbb{P}}

\newcommand{\ZZ}{\mathbb{Z}}

\newcommand{\CC}{\mathbb{C}}

\newcommand{\OO}{{\mathcal O}}

\DeclareMathOperator{\dbar}{{\overline \partial}}

\allowdisplaybreaks

\begin{document}

\title[Positive semi-definite curvature with many zeroes]{On positive semi-definite holomorphic sectional curvature with many zeroes}

\author{Minzi Chen}
\address{Minzi Chen. Department of Mathematics, University of Houston, 4800 Calhoun Road, Houston, TX 77204, USA} \email{{mchen38@math.uh.edu}}

\author{Gordon Heier}
\address{Gordon Heier. Department of Mathematics, University of Houston, 4800 Calhoun Road, Houston, TX 77204, USA} \email{{heier@math.uh.edu}}


\thanks{The second author was supported by a grant from the Simons Foundation (Grant Number 963755-GH)}

\begin{abstract}
We establish the existence of complete K\"ahler metrics of semi-positive holomorphic sectional curvature with many zeroes in an interesting and natural geometric setting. Specifically, we use Calabi's Ansatz in the form due to Koiso-Sakane to produce such metrics on the total space of powers of the tautological line bundles over the projective line. We formulate a general conjecture regarding the compact case and use a product approach to obtain an analogous result in the non-compact case.
\end{abstract}

\maketitle

\section{Introduction}

In differential geometry, a central task is to study the interaction of the curvature of a metric (e.g., a Riemannian, Hermitian, or K\"ahler metric) and the geometry of the underlying space. In one direction, the goal is to study the consequences of the existence of a metric with certain \mbox{(semi-)}definite curvature for the geometry of a, say smooth or complex manifold. Conversely, in the other direction, the goal is to establish the existence of metrics with prescribed curvature in a certain geometric setting.\par

In this paper, we address the above converse direction and work in the setting of complex manifolds and Hermitian or K\"ahler metrics. Based on the definition of the fundamental curvature 4-tensor, there are various notions of curvature that provide numerical curvature information of different “strengths.” Some classical notions are sectional curvature, bisectional curvature, holomorphic sectional curvature, Ricci curvature, scalar curvature and total scalar curvature. Here, we focus mostly on the concept of holomorphic sectional curvature, but also address Ricci curvature.\par

In the case of a K\"ahler metric, the holomorphic sectional curvature is obtained by restricting the sectional curvature of the associated Riemannian metric to complex lines, i.e., real planes invariant under the complex structure, in the tangent bundle. This notion has been studied for decades (see \cite{Tsukamoto_1957}, \cite{Grauert_Reckziegel}, \cite{Berger}, \cite{Cowen}, \cite{Wu}, \cite{Cheung} for some early milestones), but has historically been considered ``somewhat mysterious" (\cite[Remark 7.20]{Zhengbook}), especially as it relates to the Ricci curvature.\par

Around 1980, S.-T. Yau conjectured consequences of the existence of a K\"ahler metric of definite holomorphic sectional curvature for a projective or compact manifold. In the negative case, the canonical line bundle should be ample, and in the positive case, the manifold should be rationally connected or even unirational. Not much progress was made on these conjectures until the second author and his collaborators initiated a systematic program of study of these questions with modern methods around 2010. We refer the reader to the literature (e.g., \cite{heier_lu_wong_mrl}, \cite{heier_wong_cag},  \cite{HLW_JDG}, \cite{Wu_Yau_Invent},\cite{Wu_Yau_CAG}, \cite{Tosatti_Yang}, \cite{HLWZ}, \cite{Diverio_Trapani}, \cite{heier_wong_doc_math}) for more on this, as the present paper in fact concerns the converse direction, namely, we are interested in the existence problem of K\"ahler metrics of semi-positive holomorphic sectional curvature in a given geometric setting.\par
The starting point for this line of work is Hitchin's construction in \cite{Hitchin} of metrics of positive holomorphic sectional curvature on Hirzebruch surfaces $\PP(\OO_{\PP^1}\oplus\OO_{\PP^1}(k))$ using a warped metric approach to lift up the metric from the base and appropriately combine it with the fiber metrics. The second author, with his students Alvarez and Chaturvedi, and Zheng subsequently obtained more general existence results for K\"ahler metrics of positive holomorphic sectional curvature (partly with pinching) in certain geometric situations in \cite{ACH}, \cite{AHZ}, \cite{CH} based on Hitchin's approach.\par
In the paper \cite{YZ}, Yang-Zheng obtained positive holomorphic sectional curvature K\"ahler metrics on Hirzebruch manifolds, which are compactifications of split rank two bundles over projective space of arbitrary dimension, thus generalizing the notion of a Hirzebruch surface. Their main technical tool is Calabi's Ansatz \cite{Calabi_79} in the version due to Koiso-Sakane \cite{KSI}, \cite{KSII}, thus going beyond the type of metric constructed by warping. This powerful tool allows them to obtain such metrics in each K\"ahler class of a Hirzebruch manifold.\par
The aim of the present paper is to combine in the semi-positive case the approach from \cite{YZ} with a study of the phenomenon of ``unexpected zeroes"  that was discovered in the semi-negative case in \cite[Example 1.2]{HLWZ} in response to a question in \cite[Remark 1.6] {HLW_JDG}, which asked if a projective manifold with ample canonical line bundle can carry K\"ahler metric with semi-negative holomorphic sectional curvature ``with many zeroes." Here, we measure the ``amount of zeroes'' that a semi-definite Hermitian metric has in terms of the rank invariant introduced in the semi-negative case in  \cite{HLW_JDG}. 
\begin{definition}
Let $M$ be a Hermitian manifold with semi-positive (semi-negative) holomorphic sectional curvature $H$. For $p\in M$, let $\eta(p)$ be the maximum of those integers $k\in \{0,\ldots,n:=\dim M\}$ such that there exists a $k$-dimensional subspace $L\subset T_p M$ with $H(v)=0$ for all $v\in L\backslash \{0\}$. Set $\eta_M:=\min_{p\in M} \eta(p)$ and $r^+_M:=n-\eta_M$ ($r^-_M:=n-\eta_M$).
\end{definition}
For the sake of brevity, we only summarize the example from the semi-negative case here briefly as follows, referring the reader to the paper for full details. 
\begin{example}[{\cite[Example 1.2]{HLWZ}}]
Let $A$ be a principally polarized abelian variety of dimension $n+1$, and let $M\subset A$ be a theta divisor. It was proven by Andreotti-Mayer that for generic $A$, the hypersurface $M$ is non-singular, i.e., a submanifold. Take the standard flat K\"ahler metric on $A$ and restrict it to $M$. By the curvature decreasing property of subbundles, $M$ has semi-negative bisectional curvature and, in particular, $H\leq 0$. By the adjunction formula, the canonical line bundle of $M$ is ample. On the other hand, an explicit computation based on the notion of a graph metric shows that $r^-_M=\lfloor\frac{n+1}{2}\rfloor$, which is less than $n$ (when $n$ is at least $2$).
\end{example}

In this example, one of the key points is the use of the curvature decreasing property saying that if $M'$ is a submanifold of $M$, then the holomorphic sectional curvature of $M'$ does not exceed that of $M$ (see \cite{Wu}). When exploring the analogous question in semi-positive curvature, this approach is destined to fail, as the curvature decreasing property ``works in the wrong directions" in that case. Before we explain what we have been able to accomplish, we state the following conjecture, which remains open.
\begin{conjecture}\label{main_conj}
Let $n\geq 2$ be an integer. Then there exists an $n$-dimensional compact K\"ahler manifold $M$ with ample anti-canonical line bundle and semi-positive holomorphic sectional curvature such that $r^+_M= \lfloor\frac{n+1}{2}\rfloor$.
\end{conjecture}
It is the required compactness that puts this Conjecture out of reach at the moment. Instead, our goal here is to establish the existence of an appropriate complete metric as a partial result towards this conjecture. A statement of our result that matches Conjecture \ref{main_conj} closely is the following.
\begin{theorem}\label{thm_conj_form}
Let $n\geq 2$ be an integer. Then there exists an $n$-dimensional complete K\"ahler manifold $M$ with semi-positive holomorphic sectional curvature such that $r^+_M= \lfloor\frac{n+1}{2}\rfloor$. Furthermore, the nature of $M$ is such that it can be compactified to a compact complex manifold with ample anti-canonical line bundle.
\end{theorem}

We prove Theorem \ref{thm_conj_form} via a product argument based on the following Theorem \ref{main_theorem} for the total space of powers of the tautological line on the projective line. Therefore, it is perhaps the following theorem that should be considered the main theorem of this paper.
\begin{theorem}\label{main_theorem}
Let $k$ be a positive integer. The (non-compact) total space $M$ of the line bundle  $\OO(-k)\to \PP^1$ carries a complete K\"ahler metric of semi-positive holomorphic sectional curvature with $r^+_M=1$.
\end{theorem}

Following \cite{YZ}, our proof of Theorem \ref{main_theorem} is also based on the use of Calabi's Ansatz in the form due to Koiso-Sakane in this semi-positive case. The main technical difficulty is to produce a metric with ``many zeroes" which nevertheless is semi-definite. Intuitively speaking, the issue here is that a metric of semi-positive definite holomorphic sectional curvature could be made indefinite by the smallest of perturbations, while a metric of positive definite holomorphic sectional curvature would remain positive definite under small enough perturbations. Technically, our existence proof comes down to solving a certain ordinary differential equation arising from an appropriate discriminant vanishing condition. We will solve this ordinary differential equation in an elementary fashion by making an Ansatz for the solution. \par

This paper is structured as follows. In Section 2, we give basic definitions and describe Calabi's Ansatz in the form due to Koiso-Sakane. The proofs of Theorems \ref{thm_conj_form} and \ref{main_theorem} are conducted in Section 3. In Section 4, we address the case of $k=0$. Finally, in Section 5, we investigate the Ricci curvature of our metric.

\section{Definitions and preliminary material}

Let $M$ be an $n$-dimensional complex manifold and $g$ be a Hermitian metric on $M$. In local holomorphic coordinates $z_1,\ldots,z_n$, we can write
\begin{equation*}
g=\sum_{i,j=1}^n g_{i\bar j} dz_i\otimes d\bar{z}_j.
\end{equation*}
The {\it associated $(1,1)$-form} of $g$ is
$$
\omega=\frac{\sqrt{-1}}{2} \sum_{i,j=1}^n g_{i\bar j} dz_i \wedge d\bar{z}_j.
$$
It is independent of the choice of local coordinates, and is real, i.e., $\bar{\omega}=\omega$.

Let $\nabla$ be the Hermitian connection of $g$, that is, the unique connection on the tangent bundle $T_M$ which is compatible with both $g$ and the complex structure. The torsion tensor of this linear connection is given by: $T(X,Y)=\nabla_X Y-\nabla_Y X-[X,Y]$. A Hermitian metric $g$ on $M$ is called a {\it K\"ahler metric} if $T=0$. This is equivalent to $d\omega =0$.

The components $R_{i\bar j k \bar l}$ of the curvature tensor $R$ associated with the metric connection pertaining to $g$ are locally given by the formula
\begin{equation*}
R_{i\bar j k \bar l}=-\frac{\partial^2 g_{i\bar j}}{\partial z_k\partial \bar z _l}+\sum_{p,q=1}^n g^{p\bar q}\frac{\partial g_{i\bar p}}{\partial z_k}\frac{\partial g_{q\bar j}}{\partial \bar z_l},
\end{equation*}
where $g^{p\bar q}$ refers to the inverse matrix of $(g_{i\bar j})$ (see \cite[p.\ 1104, \S 1]{Wu}).

If $\xi=\sum_{i=1}^n\xi_i \frac{\partial }{\partial z_i}$ is a non-zero complex tangent vector at $p\in M$, then the {\it holomorphic sectional curvature} $H(\xi)$ is given by
\begin{equation*}
H(\xi)=\left(\sum_{i,j,k,l=1}^n R_{i\bar j k \bar l}(p)\xi_i\bar\xi_j\xi_k\bar \xi_l\right) / \left(\sum_{i,j,k,l=1}^ng_{i\bar j}g_{k\bar l} \xi_i\bar\xi_j\xi_k\bar \xi_l\right).
\end{equation*}\par
If $g$ is K\"ahler, there is only one notion of Ricci curvature (for the definitions of the first, second and third Ricci curvatures in the Hermitian case, see \cite[p.~181]{Zhengbook}), which amounts to defining the \textit{Ricci
curvature form} to be
\begin{equation*}
Ric=-\sqrt{-1}\partial\bar\partial \log\det(g_{i\bar{j}}).
\end{equation*}
Moreover, if the tangent vector $\xi=\sum_{i=1}^n\xi_i e_i$ and curvature tensor are written in term of a unitary frame $e_1,\ldots,e_n$, then 
\begin{equation*}
Ric(\xi):=Ric(\xi,\bar \xi)=\sum_{i,j,k=1}^n R_{i\bar j k \bar k}\xi_i\bar \xi_j.
\end{equation*}\par
The {\it scalar curvature} $s$ is defined to be the trace of the Ricci curvature with respect to the metric $g$, i.e., the function
\begin{equation*}
\sum_{i=1}^n Ric(e_i)=\sum_{i,j=1}^n R_{i\bar i j \bar j}.
\end{equation*}\par

Following \cite{YZ}, we now introduce the method of Calabi's Ansatz \cite{Calabi_79} in the version due to Koiso-Sakane \cite{KSI}, \cite{KSII}.\par
Let $\pi:(L,h) \to (M,g)$ be a Hermitian line bundle over a compact K\"ahler manifold $(M,g)$. We consider the $\mathbb{C}^*$-action on $L^*=L \backslash L_0$, where $L_0$ is the zero section of $L$. Denote by $V$ and $S$ the two holomorphic vector fields generated by the $\mathbb{R}^+$ and $\mathbb{S}^1$ action respectively. Denote by $\widetilde{J}$ the complex structure on $L$. Assume $t$ is a smooth function on $L^*$ which only depends on the norm of Hermitian metric $h$ on $L$. By this we mean, fix a point $v \in L^*$, under a local trivialization of $L$, we may write $v=\xi e_L$ with $\xi \neq 0$ and $h(v)=|\xi|^2 h(e_L)$, then $t$ is a single-variable function of $\sqrt{h(v)}$. Moreover we assume $t$ is strictly increasing with respect to $\sqrt{h(v)}$.

Consider a Hermitian metric on $L^*$ of the form
\begin{align}
\widetilde{g}=\pi^* g_t +dt^2+(dt \circ \widetilde{J})^2, \label{gtilde}
\end{align}
where $g_t$ is a family of Riemannian metrics on $M$ to be decided. Let $u(t)^2=\widetilde{g} (V,V)$ and it can be checked that $u$ depends only on $t$.

Let $z_1,...,z_{n-1}$ be local holomorphic coordinates on $M$ and $z_0,...,z_{n-1}$ be local coordinates on $L^*$ such that $\frac{\partial}{\partial z_0}=V-\sqrt{-1}S$.

\begin{lemma}[\cite{KSI}]
The Hermitian metric $\widetilde{g}$ defined by \eqref{gtilde} is K\"ahler on $L^*$ if and only if each $g_t$ is K\"ahler on $M$, and $g_t=g_0-U \Theta(L)$, where $U'(t)=u(t)$ and $\Theta(L)$ is the curvature form of $(L,h)$. Moreover,
\[
\widetilde{g}_{0 \bar{0}}=2 u^2, \ \widetilde{g}_{\alpha \bar{0}}=2 u \partial_{\alpha} t, \ \widetilde{g}_{\alpha \bar{\beta}}=g_{t \alpha \bar{\beta}}+2 \partial_{\alpha} t \partial_{\beta} t.
\] 
\end{lemma}

Define $p=\det(g_0^{-1} \cdot g_t)$, then $\det(\widetilde{g})=2 u^2 \cdot p \cdot \det(g_0)$. We now address the compactification at zero, which corresponds to considering K\"ahler metrics on line bundles. It is convenient to reparameterize and introduce two functions $\phi(U)=u^2 (t)$ and $Q(U)=p$. The following lemma characterizes the role of the generating function $\phi(U)$ of a smooth K\"ahler metric in the form of \eqref{gtilde} on the total space of $L \to M$.

\begin{lemma}[{\cite[Lemma 3.4]{YZ}}]
Let $(L,h)$ be a Hermitian line bundle over a compact K\"ahler manifold $(M,g_0)$, where the eigenvalues of the curvature $\Theta(L)$ with respect to $g_0$ are constant on $M$. Let $[U_{min},U_{max})$ be the range of $U$ so that $g_t=g_0-U \Theta(L)$ remains positive on $[U_{min},U_{max})$, where $U_{max}=\infty$ is allowed.

Let $\phi(U)$ be a smooth function on the interval $[U_{min},U_{max})$ which is positive on $(U_{min},U_{max})$ and satisfies $\phi(U_{min})=0$ and $\phi'(U_{min})=2$. Moreover, for any $U \in (U_{min},U_{max})$:
\begin{align}
\int_{U}^{U_{max}} \frac{dU}{\phi(U)}=+\infty.
\end{align}
Then we can solve for $t$ as a strictly increasing function of $\sqrt{h}$ which is defined on $[0,\infty)$. Moreover, we may choose $t_{min}>-\infty$ where $(t_{min}, t_{max})$ is the range of $t$, and the correspoding $\widetilde{g}$ in \eqref{gtilde} is a smooth K\"ahler metric on the total space of $L$.
\end{lemma}

We will be working here with $L=\OO(-1)$ being the tautological line bundle, i.e., the dual of the hyperplane line bundle, on $\mathbb{P}^{n-1}$. Naturally, we consider its $k$-th power $L^k \to \mathbb{P}^{n-1}$ for a positive integer $k$. Here we assume the base $\mathbb{P}^{n-1}$ is endowed with the Fubini-Study metric $Ric(g_0)=g_0$, hence
\begin{align}
R_{i \bar{j} k \bar{l}}(g_0)=\frac{1}{n}\{(g_0)_{i \bar{j}} (g_0)_{k \bar{l}}+(g_0)_{k \bar{j}} (g_0)_{i \bar{l}}\}. \label{Rijkl}
\end{align}
It follows that
\begin{align}
\widetilde{g}_{i \bar{j}}=(1+\frac{k}{n} U)(g_0)_{i \bar{j}}, \ \Theta(L^k)=-\frac{k}{n}(g_0)_{i \bar{j}}, \ t_{i \bar{j}}=\frac{k}{2n}u (g_0)_{i \bar{j}}. \label{gTHtij}
\end{align}
Here we need to ensure that $U_{min}>-\frac{n}{k}$ and $U_{max}=\infty$.\par
The following lemma gives us the components of the curvature tensor of such a metric $\widetilde{g}$ in a unitary frame.
\begin{lemma}[{\cite[Proposition 3.6]{YZ}}]\label{YZ3.6}
Let $p \in \mathbb{P}^{n-1}$. We assume that $\partial_{\alpha} t=\partial_{\bar{\alpha}} t=0$ $(1 \leq \alpha \leq n-1)$ on the fiber $\pi^{-1}(p)$ of $\pi:L^k \to \mathbb{P}^{n-1}$. Consider the unitary frame ${e_0,e_1,...,e_{n-1}}$ along this fiber:
\begin{align}
e_0=\frac{1}{\sqrt{2 \phi}} \frac{\partial}{\partial z_0}, \ e_i=\frac{1}{\sqrt{(1+\frac{k}{n}U)(g_0)_{i \bar{i}}}} \frac{\partial}{\partial z_i} (1 \leq i \leq n-1). \label{e0ei}
\end{align}
Then the only nonzero curvature components of $\widetilde{g}$ on the total space of $L^k$ are:
\begin{align}
A &:= \widetilde{R}_{0 \bar{0} 0 \bar{0}}=-\frac{1}{2} \frac{d^2 \phi}{d U^2}, \label{AR0000}\\
B &:= \widetilde{R}_{0 \bar{0} i \bar{i}}=\frac{k^2 \phi-k(n+kU) \frac{d \phi}{dU}}{2(n+kU)^2}, \label{BR00ii}\\
C &:= \widetilde{R}_{i \bar{i} i \bar{i}}=2 \widetilde{R}_{i \bar{i} j \bar{j}}=\frac{[2(n+kU)-k^2 \phi]}{(n+kU)^2}, \label{CRiiii}
\end{align}
where $1 \leq i,j \leq n-1$ and $i \neq j$.
\end{lemma}

\section{Proofs of Theorem \ref{main_theorem} and Theorem \ref{thm_conj_form}}

Lemma \ref{YZ3.6} leads to the following characterization of $U(n)$-invariant K\"ahler metrics with positive holomorphic sectional curvature via the discriminant criterion for zero-free quadratic real polynomials.
\begin{proposition}[{\cite[Proposition 3.7]{YZ}}]\label{YZ3.7inf}
Any $U(n)$-invariant K\"ahler metric on the total space of $L^k \to \mathbb{P}^{n-1}$ has positive holomorphic sectional curvature if and only if
\begin{equation}
A>0,\quad C> 0,\quad 2B +\sqrt{AC} >0. 
\end{equation}
In other words, it is characterized by a smooth function $\phi(U)$ on $[U_{min},+\infty)$ for some $U_{min}>-\frac{n}{k}$ such that the following conditions hold:
\begin{enumerate}
\item  $\phi>0$ and $\phi''<0$ on $(U_{min},+\infty)$, $\phi(U_{min})=0$, $\phi'(U_{min})=2$.

\item $\int_{U}^{+\infty} \frac{1}{\phi(U)} dU=+\infty$ for any $U \in [U_{min},+\infty)$.

\item $\phi(U)<\frac{2}{k^2}(n+kU)$, and $\frac{k \phi}{n+kU}-\phi'>-\sqrt{(-\frac{1}{2} \phi'')(\frac{2}{k^2}(n+kU)-\phi)}$ for any $U \in [U_{min},+\infty)$.
\end{enumerate}
\end{proposition}
We remark that ``Condition (iii) and $\phi''<0$'' is equivalent to $H > 0$ on any tangent space. In the proof of \cite[Proposition 3.7]{YZ}, the expression of the holomorphic sectional curvature is given as
$$
H(X)=A |x_0|^4+4B |x_0 x_1|^2+C |x_1|^4,
$$
where $X=x_0 e_0+x_1 e_1$ is a (1,0)-vector at any point with $|x_0|^2+|x_1|^2=1$. When $|x_0|^2=0$ it is clear that $H(X)>0$ is equivalent to $C>0$. In the case $|x_0|^2>0$, $H(X)>0$ is equivalent to
$$
h(y)=A+4B y+C y^2>0,
$$
where $y=\frac{|x_1|^2}{|x_0|^2} \in [0,+\infty)$. First let $h(0)>0, \lim\limits_{y \to +\infty} h(y)>0$, we have $A,C>0$. If $B>0$, all the coefficients are positive; if $B \leq 0$, let $\Delta=16 B^2-4AC<0$, we get $(-2B)^2<AC$. Therefore, in terms of curvature components in Lemma \ref{YZ3.6}, ``Condition (iii) and $\phi''<0$ '' reads $A,C>0$ and $2B+\sqrt{AC}>0$.

For the semi-positive case with a zero curvature value, we set $2B+\sqrt{AC}=0$ and thus change the second inequality in (iii) into equality. We then determine $\phi$ satisfying the (new) conditions.

Set
\begin{align}
\frac{k \phi}{n+kU}-\phi'=-\sqrt{(-\frac{1}{2} \phi'')(\frac{2}{k^2}(n+kU)-\phi)}. \label{3.7iiieq}
\end{align}
We observe that if we set $\phi(U)=(n+kU) \psi(U)$, then
\begin{align*}
\phi'(U)&=k \psi(U)+(n+kU) \psi'(U), \\
\phi''(U)&= 2k \psi'(U)+(n+kU) \psi''(U),
\end{align*}
and the equation becomes
$$
(n+kU) \psi'=\sqrt{\frac{1}{2} (2k \psi'+(n+kU) \psi'') (n+kU) (\psi-\frac{2}{k^2})}.
$$
Since $n+kU>0$, we have $\psi' \geq 0$. We need to find a function $\psi$ satisfying:
\begin{align}
2 \psi'^2=(\frac{2k}{n+kU} \psi'+\psi'')(\psi-\frac{2}{k^2}). \label{psieq}
\end{align}
Note that there is a fraction with denominator $n+kU$ on the right hand side of the equation and we are thus led to making another Ansatz of the form:
$$
\psi(U)=\frac{a+bU}{c+dU} \quad (b,d \neq 0).
$$
Without loss of generality, we can rewrite the fraction
$$
\frac{a+bU}{c+dU}=\frac{a'+b' U}{c'+U}
$$
by letting $a'=\frac{a}{d},b'=\frac{b}{d},c'=\frac{c}{d}$. So the number of the ``degrees of freedom'' in the fraction in the Ansatz is three. Nevertheless, we will be using $a,b,c,d$ in the following discussion. The first and second derivatives of $\psi$ are:
\begin{align*}
\psi'(U)&=\frac{bc-ad}{(c+dU)^2},\\
\psi''(U)&=\frac{2d(ad-bc)}{(c+dU)^3}.
\end{align*}
We substitute $\psi,\psi',\psi''$ above into \eqref{psieq} to obtain:
$$
2(\frac{bc-ad}{(c+dU)^2})^2=(\frac{2k(bc-ad)}{(n+kU)(c+dU)^2}+\frac{2d(ad-bc)}{(c+dU)^3})(\frac{a+bU}{c+dU}-\frac{2}{k^2}),
$$
which is equivalent to
\begin{align*}
&2(\frac{bc-ad}{(c+dU)^2})^2-(\frac{2k(bc-ad)}{(n+kU)(c+dU)^2}+\frac{2d(ad-bc)}{(c+dU)^3})(\frac{a+bU}{c+dU}-\frac{2}{k^2})\\
=\ &\frac{2(ad-bc)(a k^3-b n k^2-2 c k+2 d n)}{k^2 (n+kU)(c+dU)^3}=0.
\end{align*}
Note that if $ad-bc=0$, $\psi \equiv \frac{b}{d}$ is a constant function and $\phi''(U)=0$ does not satisfy (i) in Proposition \ref{YZ3.7inf}. Therefore $ad-bc \neq 0$ and
\begin{align}
a k^3-b n k^2-2 c k+2 d n=0. \label{cond1}
\end{align}
Setting $\phi(U)=0$, the two solutions are $-\frac{n}{k}$ and $-\frac{a}{b}$. Since $U_{min}>-\frac{n}{k}$, we have $U_{min}=-\frac{a}{b}$. The denominator of $\phi(U_{min})$ is $c+dU_{min}=-\frac{ad-bc}{b} \neq 0$. Moreover, the condition $\phi'(U_{min})=2$ gives:
\begin{align}
\phi'(-\frac{a}{b})=\frac{b(ak-bn)}{ad-bc}=2. \label{cond2}
\end{align}
Now there are two equations. Considering the left hand side of \eqref{cond1} as a polynomial of $n,k$ with all coefficients 0, let $t_1$ be the new parameter and
\begin{align*}
a(n,k)&=a_0(n,k)+a_1(n,k) t_1,\\
b(n,k)&=b_0(n,k)+b_1(n,k) t_1,\\
c(n,k)&=c_0(n,k)+c_1(n,k) t_1,\\
d(n,k)&=d_0(n,k)+d_1(n,k) t_1.
\end{align*}
We set $a_0(n,k)=2n$, \ $b_0(n,k)=2k$ to cancel out the terms with degree $>1$ and $b_1(n,k)=d_1(n,k)=0$ to deal with $n$. Then we substitute $a,b,c,d$ into \eqref{cond1} and \eqref{cond2}:
\begin{align*}
(2n d_0-2k c_0)+(k^3 a_1-2k c_1)t_1&=0,\\
(4k c_0-4n d_0)+(2k^2 a_1-2 a_1 d_0+4k c_1)t_1&=0.
\end{align*}
Then $a_1,c_0,c_1,d_0$ satisfy:
\begin{align}
k c_0-n d_0&=0,\label{c0d0}\\
k^3 a_1-2k c_1&=0,\label{a1c1}\\
2k^2 a_1-2 a_1 d_0+4k c_1&=0.\label{a1c1d0}
\end{align}
First, substituting \eqref{a1c1} into \eqref{a1c1d0}, we get
$$
2a_1 (k^2+k^3-d_0)=0.
$$
If $a_1=0$, then $c_1=0$ by \eqref{a1c1} and we let $c_0=n t_2, d_0=k t_2$ so that \eqref{c0d0} is satisfied. However, $ad-bc=(2n)(k t_2)-(2k)(n t_2)=0$ so there is a contradiction. Since $a_1 \neq 0$, we can solve for:
\begin{align*}
c_0&=nk(k+1),\\
d_0&=k^2 (k+1),\\
c_1&=\frac{k^2}{2} a_1.
\end{align*}
Setting $a_1=2$, we obtain $c_1=k^2$ and
\begin{equation}\label{phiequals}
\phi(U)=\frac{2(n+kU)(n+t_1+kU)}{k((k+1)(n+kU)+k t_1)}.
\end{equation}

Now we need to determine one constant $t_1$ such that $U_{min}=-\frac{n+t_1}{k}>-\frac{n}{k}$. It is equivalent to $t_1<0$. We now check the remaining conditions from the items (i)-(iii) in Proposition \ref{YZ3.7inf}: 
\begin{enumerate}
\item[(i)] When $U>U_{min}=-\frac{n+t_1}{k}$,
$$(k+1)(n+kU)+k t_1>k(n+kU)+kt_1=k(n+kU+t_1)>0,$$
so $\phi(U)>0$.\par

Differentiate $\phi$ twice,
$$
(\frac{(n+kU)(a+bU)}{c+dU})''=-\frac{2(ad-bc)(kc-nd)}{(c+dU)^3},
$$
where
\begin{align*}
ad-bc&=(2n+2t_1)(k^2(k+1))-(2k)(nk(k+1)+k^2 t_1)=2k^2 t_1,\\
kc-nd&=k(nk(k+1)+k^2 t_1)-n(k^2(k+1))=k^3 t_1.
\end{align*}
Then
$$\phi''(U)=-\frac{4 k^2 t_1^2}{((k+1)(n+kU)+k t_1)^3}<0.$$

\item[(ii)] For any $\tilde{U} \geq U_{min}$,
\begin{align*}
\int_{\tilde{U}}^{+\infty} \frac{1}{\phi(U)} dU&=\int_{\tilde{U}}^{+\infty} \frac{k((k+1)(n+kU)+k t_1)}{2(n+kU)(n+kU+t_1)} dU\\
&=\int_{\tilde{U}}^{+\infty} \frac{k}{2} \cdot \frac{k}{n+kU}+\frac{1}{2} \cdot \frac{k}{n+kU+t_1} dU\\
&=\frac{1}{2} (k \log(n+kU)+\log(n+kU+t_1))|^{+\infty}_{\tilde{U}}=+\infty.
\end{align*}

\item[(iii)] Evaluate the difference:
\begin{align*}
\phi(U)-\frac{2}{k^2}(n+kU)&=(n+kU)(\frac{a+bU}{c+dU}-\frac{2}{k^2})\\
&=(n+kU)(\frac{(k^2 a-2c)+(k^2 b-2d)U}{k^2 (c+dU)})\\
&=-\frac{2(n+kU)^2}{k^2 ((k+1)(n+kU)+k t_1)}<0.
\end{align*}
\end{enumerate}

One can sum up the above by stating that we have found the family of functions
$$
\{\phi_{t_1}(U)=\frac{2(n+kU)(n+kU+t_1)}{k((k+1)(n+kU)+k t_1)}:t_1<0\},
$$
yielding a K\"ahler metric with semi-positive holomorphic sectional curvature on the total space of $L^k \to \mathbb{P}^{n-1}$ for every negative $t_1$. By construction, the metric satisfies $r^+_M<n$. This puts us in a position to prove Theorem \ref{main_theorem}.

\begin{proof}[Proof of Theorem \ref{main_theorem}]
We apply the above discussion in the case $n=2$. It is a basic fact that the holomorphic sectional curvature of a K\"ahler metric completely determines the curvature tensor $R$ (\cite[Proposition 7.1,~p.~166]{kobayashi_nomizu_ii}). In particular, if $H$ vanishes identically, then $R$ vanishes identically. This rules out the case that $r^+_M=0$. We know by construction that $r^+_M<n=2$, and thus we are left with $r^+_M=1$.\par
In order to prove the completeness of the metric, by the Theorem of Hopf-Rinow, it suffices to establish that closed and bounded sets are compact. To this end, for a real number $s > 0$, let
$$
c:(0,s) \to L, t \mapsto (t,0)
$$
be a real curve in the above local coordinates $z_0,z_1$. Then
\begin{align*}
\int_0^s g(c'(t),c'(t)) dt = \int_0^s 2u^2dt
\end{align*}
which for $s$ going to infinity is a divergent integral based on \eqref{phiequals}. This ensures that closed and bounded sets with respect to the metric are also bounded in the standard distance on the fibers and thus compact by the Heine-Borel Theorem.
\end{proof}
\begin{remark}
The completeness argument for the metric in the proof of Theorem \ref{main_theorem} carries over verbatim to the case of arbitrary $n$.
\end{remark}
By computing the holomorphic sectional curvature of the product metric on the product of Hermitian manifolds as it was done in the proof of \cite[Theorem 1.2]{ACH}, we can find the relationship between the ``amount of zeroes'' of the curvature of Hermitian manifolds with semi-positive holomorphic sectional curvatures before and after taking the product. In fact, in the following lemma, we find the expected {\it addition formula} to hold.

\begin{lemma}\label{prod_lemma}
Let $M$ and $N$ be Hermitian manifolds with semi-positive holomorphic sectional curvatures $H_M$ and $H_N$ respectively. Let $M \times N$ be the product manifold equipped with the product metric and holomorphic sectional curvature $H_{M \times N}$. Then $r^+_{M \times N}=r^+_M+r^+_N$.
\end{lemma}

\begin{proof}
Let $p \in M, u \in T_p M, q \in N, v \in T_q N$ with $H_M(u)=H_N(v)=0$. Let $w=au+bv \in T_{(p,q)} (M \times N)$. By the formulas in the top half of \cite[p.~141]{ACH} it is clear that $H_{M \times N}(w)=0+0=0$. This yields the inequality $r^+_{M \times N} \leq r^+_M+r^+_N$.

On the other hand, let $w \in T_{(p,q)} (M \times N)$ where $p \in M, q \in N$. Choose $u \in T_p M, v \in T_q N$ such that $w=u+v$. Since $H_M, H_N, H_{M \times N}$ are semi-positive, $H_{M \times N}(w)=0$ implies $H_M(u)=H_N(v)=0$. Then $r^+_{M \times N} \geq r^+_M+r^+_N$.
\end{proof}
\begin{remark}
Note that it is in general not true that $H_{M \times N}(u+v)= H_M(u)+H_N(v)$ for arbitrary $u \in T_p M$ and $v \in T_q N$ without the vanishing assumption. However, the formulas on \cite[p.~141]{ACH} do allow the proof of Lemma \ref{prod_lemma} to go through.
\end{remark}
We are now in a position to prove Theorem \ref{thm_conj_form} via a product construction, using the total space of $L^k \to \mathbb{P}^{1}$ with $k=1$ as building block.
\begin{proof}[Proof of Theorem \ref{thm_conj_form}]
Let $M_2$ be the total space of $L^1 \to \mathbb{P}^{1}$ with $r^+_{M_2}=1$ as obtained in the proof of Theorem \ref{main_theorem}. Recall that the natural compactification of $M_2$ is the first Hirzebruch surface, which has ample anti-canonical line bundle.

Recall that for $\mathbb{P}^1$ with the Fubini-Study metric, the holomorphic sectional curvature is a positive constant. So $r^+_{\mathbb{P}^1}=1$.

If $n \geq 2$ is an even number, let $M$ be the product of $\frac{n}{2}$ copies of $M_2$:
$$
M=M_2 \times ... \times M_2,
$$
with the metric on $M$ being the product metric of the metrics obtained in the proof of Theorem \ref{main_theorem}.
By repeated application of the addition formula from Lemma \ref{prod_lemma}, we find $r^+_M=\frac{n}{2} \cdot 1=\lfloor\frac{n+1}{2}\rfloor$. Furthermore, $M$ is also complete and can be compactified to a compact complex manifold with ample anti-canonical line bundle since this is true for the factors and is preserved under taking the product.

If $n \geq 3$ is odd, let $M$ be the product of $\frac{n-1}{2}$ copies of $M_2$ and one copy of $\mathbb{P}^1$ with the obvious product metric. By the addition formula, $r^+_M=\frac{n-1}{2} \cdot 1+1=\frac{n+1}{2}=\lfloor\frac{n+1}{2}\rfloor$ since $n$ is odd. Again, completeness and ample anti-canonical line bundle under compactification are preserved.
\end{proof}
\section{The case of $k=0$}
Due to its reliance on Lemma \ref{YZ3.6}, the proof of Theorem \ref{main_theorem} does not go through in the untwisted case $k=0$. In order to briefly address this case, we would like to address it with a conformal change approach. We let $z_1$ be an inhomogeneous coordinate on an open subset of $\PP^1$ and $z_2$ a coordinate on $\CC$. We consider the Fubini-Study product metric whose associated $(1,1)$-form is given by
$$
\omega=\frac{\sqrt{-1}}{2} \partial \dbar (\log(1+z_1 \bar{z}_1)+\log(1+z_2 \bar{z}_2))
$$
and let $\omega_f=e^f \omega$, where $f=f(z_2 \bar{z}_2)$.\par
By an explicit computation, we see that the numerator of the holomorphic sectional curvature of a tangent vector $\xi=(u,v)$ is
$$
H(\xi)=R_{1 \bar 1 1 \bar 1} |u|^4+R_{1 \bar 1 2 \bar 2} |u|^2 |v|^2+R_{2 \bar 2 2 \bar 2} |v|^4.
$$
Having $H(\xi)\geq 0$ for all tangent vectors $\xi$ and with $H$ not identically zero but at least one tangent vector $\xi_0= (u,v)$ with $H(\xi_0)=0$ is implied by $R_{1 \bar 1 1 \bar 1} > 0$ (which is true) and 
the discriminant condition
\begin{equation*}
R_{1 \bar 1 2 \bar 2}+2 \sqrt{R_{1 \bar 1 1 \bar 1} R_{2 \bar 2 2 \bar 2}}=0,
\end{equation*}
which is equivalent to
\begin{equation*}
-(f'(r)+r f''(r))+2 \sqrt{2} \sqrt{\frac{2}{(1+r)^4}-\frac{f'(r)+r f''(r)}{(1+r)^2}}=0,
\end{equation*}
where $r=z_2 \bar{z}_2$. We solve the ODE by making an Ansatz that
$$
h(r):=rf'(r)=\frac{A}{1+r}+C
$$
with real constants $A,C$, and obtain the conformal factor
$$
e^{f(r)}=e^{C_2} \frac{(1+r)^{4(\sqrt{2}-1)}}{r^{4(\sqrt{2}-1)-C}}.
$$
We see that $\lim_{r\to 0^+} e^{f(r)}$ and $\lim_{r\to \infty}e^{f(r)}$ cannot both have a finite non-zero value at the same time. This precludes us from extending the conformally changed Hermitian metric to a metric on the compact manifold $\PP^1\times \PP^1$.\par
However, choosing $C=4(\sqrt{2}-1)$ and $C_2=0$, gives us
\begin{equation*}
e^{f(r)}=(1+r)^{4(\sqrt{2}-1)}
\end{equation*}
which yields a well-defined metric on $\PP^1\times \CC$ which is complete. However, it is easy to check that this Hermitian metric is not K\"ahler.

\section{Ricci curvature}

For the metric characterized by \eqref{phiequals} on the total space of $L^k \to \mathbb{P}^{n-1}$, the Ricci curvature of a unit (1,0)-vector $X=x_0 e_0+x_1 e_1$ is:
$$
Ric(X)=(A+(n-1)B) |X_0|^2+(B+\frac{n}{2}C) |X_1|^2,
$$
where
\begin{align*}
A+(n-1)B&=\frac{2 k^2 t_1^2+(n-1)k t_1 ((k+1)(n+kU)+k t_1)}{((k+1)(n+kU)+k t_1)^3},\\
B+\frac{n}{2}C&=\frac{k t_1+n ((k+1)(n+kU)+k t_1)}{((k+1)(n+kU)+k t_1)^2}.
\end{align*}

For sufficiently large $U$, we note 
$$A+(n-1)B<0<B+\frac{n}{2}C,$$
making the Ricci curvature clearly not definite and even not semi-definite. However, we can choose $n,k$ such that the coefficients have the same sign on a subset of $[t_1, +\infty)$. Consequently, the Ricci curvature is positive-definite or negative-definite in some regions on the manifold.\par

Somewhat surprisingly, for the unit vector $X=X_0 e_0+X_1 e_1$, if the holomorphic sectional curvature $H(X)=0$, then the Ricci curvature $Ric(X)=-B>0$ is always positive.\par

Due to the indefiniteness of the Ricci curvature, we would like to close by briefly presenting a more delicate discussion in terms of the notion of positivity for sums of eigenvalues of the Ricci curvature of a K\"ahler manifold in the sense of the following definition.

\begin{definition}
For the purpose of this definition, we write the Ricci curvature form for the metric characterized by \eqref{phiequals} as:
$$
Ric=\sqrt{-1}\sum\limits_{i,j=0}^{n-1}\tilde{R}_{ij} e_i^* \wedge \bar e_j^*,
$$
where $(\tilde{R}_{ij})_{i,j=0}^{n-1}$ is a Hermitian $n \times n$ matrix and $\{e_0^*,\ldots,e_{n-1}^*\}$ is the dual basis of $\{e_0,\ldots,e_{n-1}\}$. The Ricci curvature is said to be $\kappa$-{\it positive} at a point $p \in M$ if the eigenvalues of $(\tilde{R}_{ij})_{i,j=0}^{n-1}$ have the property that any sum of $\kappa$ of them is positive.
\end{definition} 

\begin{definition}
Depending on the above parameters $n,k,U$, we define the set
\begin{align*}
K(n,k,U) =\{\kappa\in  \{1,\ldots,n\}: & \text{ the Ricci curvature is }\\
&\ \kappa\text{-}\text{positive at points corresponding to }U\}.
\end{align*}
\end{definition}
In the following proposition, we first observe that these sets are never empty, as they always contain $n$.
\begin{proposition}\label{prop_kappa_n}
For any $n,k,U$, we have $n \in K(n,k,U)$.
\end{proposition}
\begin{proof}
For the metric characterized by \eqref{phiequals}, the eigenalues are $A+(n-1)B$ (of multiplicity 1) and $B+\frac{n}{2}C$ (of multiplicity $n-1$), so the Ricci curvature is $n$-positive if
\begin{equation*}
A+(n-1)B + (n-1)(B+\frac{n}{2}C) > 0.
\end{equation*}
This positivity can be verified by an explicit calculation, similar to the argument in the proof of Proposition \ref{kappa2}. Alternatively, one can observe that the sum of all the eigenvalues of the Ricci curvature is nothing but the scalar curvature and cite a result of Berger \cite[Lemme 7.4]{Berger} which yields that positivity of the holomorphic sectional curvature of a K\"ahler metric implies the positivity of the scalar curvature. 
\end{proof}

\begin{proposition}\label{kappa2}
For any $n,U$, we have $2 \in K(n,1,U)$.
\end{proposition}
\begin{proof}
When $n=2$, by Proposition \ref{prop_kappa_n}, we have $2 \in K(2,1,U)$.
So we may assume that $n \geq 3$. In order for the metric characterized by \eqref{phiequals} to be $\kappa$-positive for $\kappa < n$, we need the inequalities
$$
B+\frac{n}{2}C>0, \quad A+(n-1)B+(\kappa-1)(B+\frac{n}{2}C)>0
$$
to hold. It actually suffices to have $n \geq 2=k+1$ for the first inequality to hold. The left hand side of the second inequality can be written as
$$
A+(n-1)B+(\kappa-1)(B+\frac{n}{2}C)=\frac{D_0+D_1(\kappa) V+D_2(\kappa) V^2}{V^3},
$$
where $V=(k+1)(n+kU)+k t_1 \geq -t_1>0$ and
$$
D_0=2k^2 t_1^2, \quad D_1(\kappa)=(n+\kappa-2)k t_1, \quad D_2(\kappa)=n(\kappa-1).
$$
In our present case of $k=1, \kappa=2$, the numerator
$$
D_{(\kappa)}(V):=D_0+D_1(\kappa) V+D_2(\kappa) V^2=2 t_1^2+n t_1 V+n V^2,
$$
and $\inf\limits_{V \geq -t_1} D_{(2)}(V)=D_{(2)}(-t_1)=2 t_1^2>0$.
\end{proof}
The following proposition states a well-known general fact about $\kappa$-positive Ricci curvature of a K\"ahler manifold based on a simple combinatorial addition formula. However, we merely state the fact here for the case we are interested in.
\begin{proposition}
If $\kappa \in K(n,k,U)$, then $\kappa+1 \in K(n,k,U)$.
\end{proposition}
\begin{proof}
The sum of any $\kappa+1$ eigenvalues $\lambda_1,\ldots\lambda_{\kappa+1}$ satisfies the identity
$$
\sum\limits_{j=1}^{\kappa+1} \lambda_j=\frac{1}{\kappa} \sum\limits_{k=1}^{\kappa+1}(\sum\limits_{j=1\atop j \neq k}^{\kappa+1} \lambda_j).
$$
The expression on the right hand side is clearly positive due to $\kappa \in K(n,k,U)$, as the interior sum is only a sum of $\kappa$ eigenvalues. 
\end{proof}
The converse is not true for arbitrary $\kappa$, and it seems interesting to determine the smallest element of the set $\bigcap\limits_{U \geq -t_1} K(n,k,U)$, which we denote by $\kappa'$. For any $n,k$, the smallest element $\kappa' \neq 1$ because $A+(n-1)B<0$ for sufficiently large $U$. We have already found $\kappa'=2$ in the $k=1$ case. When $k \geq 2$ and $n \leq k$, there exists $U_0 \geq -t_1$ such that $B+\frac{n}{2} C \leq 0$, so in this case $\kappa \in \{1,\ldots,n-1\}$ implies $\kappa \notin \bigcap\limits_{U \geq -t_1} K(n,k,U)$. Under the assumption $n \geq k+1$ (and thus $B+\frac{n}{2} C>0$ for all $U \geq -t_1$), we are able to determine the precise value of $\kappa'$ as in the following proposition.
\begin{proposition}\label{kappaprime}
For $k \geq 2, n \geq k+1$, we have
$$
\kappa'=\lfloor 3n+3-\sqrt{8n(n+1)} \rfloor.
$$
\end{proposition}
\begin{proof}
Since $\kappa' \neq 1$, we start by letting $\kappa'=2$. We need to find $n,k \in \ZZ^+$ with $n \geq k+1 \geq 3$ such that
\begin{align*}
2 \in \bigcap\limits_{U \geq -t_1} K(n,k,U)&\iff D_{(2)}(V)>0, \forall V \geq -t_1\\
&\iff D_{(2)}(-kt_1)=(2-\frac{n}{4})k^2 t_1^2>0\\
&\iff  n<8.
\end{align*}

Assume $3 \leq \kappa' \leq n$. Then we are seeking $n,k \in \ZZ^+$ with $n \geq k+1 \geq 3$ such that
\begin{align*}
\kappa' \in \bigcap\limits_{U \geq -t_1} K(n,k,U) &\iff \inf\limits_{V \geq -t_1} D_{(\kappa')}(V)>0\\
&\iff (n+\kappa'-2)^2<8n (\kappa'-1),\\
\kappa'-1 \notin \bigcap\limits_{U \geq -t_1} K(n,k,U) &\iff \inf\limits_{V \geq -t_1} D_{(\kappa'-1)}(V) \leq 0\\
&\iff (n+\kappa'-3)^2 \geq 8n(\kappa'-2).
\end{align*}
For given $\kappa' \geq 3$, we can solve the inequalities for $n$:
$$
3 \kappa'-5+\sqrt{8 (\kappa'-1)(\kappa'-2)} \leq n < 3 \kappa'-2+\sqrt{8 \kappa'(\kappa'-1)}.
$$
And we can solve for $\kappa'$, the solution is
\begin{align*}
&3n+2-\sqrt{8n(n+1)} < \kappa' \leq 3n+3-\sqrt{8n(n+1)}\\
\iff &\kappa'=\lfloor 3n+3-\sqrt{8n(n+1)} \rfloor.
\end{align*}
This is also true for the $\kappa'=2$ case. Let $3 \leq n <8$, then
$$
\lfloor 3n+3-\sqrt{8n(n+1)} \rfloor=2.
$$
\end{proof}

\begin{remark}
Interestingly enough, one can apply L'H\^{o}pital's rule to find that $\kappa'=(3-2 \sqrt{2})n+o(n)$ in Proposition \ref{kappaprime}. Note that $3-2 \sqrt{2} \approx 0.17$.
\end{remark}
Finally, we summarize our findings in the form of the following remark.
\begin{remark}
We write
$$
K(n,k) =\bigcap\limits_{U \geq -t_1} K(n,k,U).
$$
Then $\max K(n,k)=n$ and
$$
\kappa'=\min K(n,k)=
\begin{cases}
2, \quad k=1\\
\lfloor 3n+3-\sqrt{8n(n+1)} \rfloor, \quad k \geq 2, n \geq k+1\\
n, \quad k \geq 2, 2 \leq n \leq k.
\end{cases}
$$
\end{remark}

\begin{acknowledgement}
This article is based on the first author's Ph.D. dissertation, written under the direction of the second author at the University of Houston.
\end{acknowledgement}

\end{document}